\documentclass[12pt]{amsart}
\usepackage{amssymb}
\usepackage{amssymb}
\usepackage{amssymb}
\usepackage{amssymb}
\usepackage{amssymb}
\usepackage{amsfonts}
\usepackage{amssymb}
\usepackage{amssymb}
\usepackage{amssymb}
\usepackage{amssymb}
\usepackage{amssymb}
\usepackage{amssymb}
\usepackage{amssymb}
\usepackage{amssymb}
\usepackage{amssymb}
\usepackage{amssymb}
\usepackage{amssymb}
\usepackage{amssymb}
\usepackage{amsfonts}
\usepackage{amssymb}
\usepackage{amssymb}
\usepackage{amssymb}
\usepackage{amssymb}
\usepackage{amssymb}
\usepackage{amsfonts}
\usepackage{amsfonts}
\usepackage{amsfonts}
\usepackage{amssymb}
\usepackage{amsfonts}
\usepackage{amsfonts}
\usepackage{amsfonts}
\usepackage{amsfonts}
\usepackage{amsfonts}
\usepackage{amsfonts}
\usepackage{amsfonts}
\usepackage{amsfonts}
\usepackage{amsfonts}
\usepackage{amsfonts}
\usepackage{amssymb}
\usepackage{amsfonts}
\usepackage{amsfonts}
\usepackage{amsfonts}
\usepackage{amssymb}
\usepackage{amscd}
\usepackage{amssymb}
\usepackage{amssymb}
\usepackage{amssymb}
\usepackage{amssymb}
\usepackage{amssymb}
\usepackage{amssymb}
\usepackage{amssymb}
\usepackage{amssymb}
\usepackage{amssymb}
\usepackage{amssymb}
\usepackage{amssymb}
\usepackage{amssymb}
\usepackage{amssymb}
\usepackage{amssymb}
\usepackage{cases}
\usepackage{amsmath}
\usepackage{amsfonts}
\usepackage{mathrsfs}
\usepackage{amscd,amsmath,amssymb,amsfonts}
\usepackage[all]{xy}
\theoremstyle{plain}
\newtheorem{thm}{Theorem}
\newtheorem{lem}[thm]{Lemma}

\theoremstyle{definition}
\newtheorem{defn}[thm]{Definition}

\newtheorem{rmk}[thm]{Remark}

\newtheorem{construction}[thm]{Construction}
\numberwithin{thm}{section} \numberwithin{equation}{section}

\newcommand{\sO}{{\mathcal O}}

\begin{document}

\title{The H-N filtration of bundles as Frobenius pull-back}
\author{Mingshuo Zhou}
\address{Academy of Mathematics and Systems Science, Chinese Academy of Science,
Beijing, P. R. of China} \email{zhoumingshuo@amss.ac.cn}
\date{April 20, 2012}
\thanks{The author is supported by gucas.}

\begin{abstract} Let $X$ be a smooth projective curve of genus {$g\geq 2$} over an
algebraic closed field $k$ of characteristic $p>0$. Let $F:
X\rightarrow X_1$ be the relative Frobenius morphism, and $E$ be a
semistable torsion free sheaf on $X$. For a semistable vector bundle
$E$, one may guess that the length of Harder-Narasimhan filtration
is not more than $p$. In this paper, I give a negative answer to
this by giving an explicit example.
\end{abstract}
\maketitle
\begin{quote}

\end{quote}

\section{introduction}
Let $X$ be a smooth projective curve of genus $g\geq 2$ defined over
an algebraic closed field $k$ of characteristic $p>0$. The absolute
Frobenius morphism $F_X: X\rightarrow X$ is induced by
$\sO_X\rightarrow \sO_X, f\mapsto f^p$. Let $F: X\rightarrow X_1:=
X\times_k k$ denote the relative Frobenius morphism over $k$. One of
the themes is to study its action on the geometric objects on $X$.

Recall that a vector bundle $E$ on a smooth projective curve is
called semistable (resp. stable) if $\mu (E')\leq \mu(E)$ (resp.
$\mu(E')< \mu(E)$) for any nontrivial proper subbundle $E'\subset
E$, where $\mu(E)$ is the slope of $E$. Semistable bundles are basic
constituents of vector bundle in the sense that any bundle $E$
admits a unique filtration $$HN_{\bullet}(E): 0=HN_0(E)\subset
HN_1(E)\subset \cdots \subset HN_{\ell}(E)=E'$$ which is the so
called Harder-Narasimhan filtration, such that

\noindent (1) $gr_i^{HN}(E):= HN_i(E)/HN_{i-1}(E) (1\leq i\leq
\ell)$ are semistable;

\noindent (2) $\mu(gr_1^{HN}(E))>\mu(gr_2^{HN}(E))>\cdots
>\mu(gr_{\ell}^{HN}(E))$.
The integer $\ell$ is called the length of the Harder-Narasimhan
filtration of bundle $E$. It measures how far is a vector bundle
from being semistable in some sense, and it's clear $E$ is
semistable if and only if $\ell =1$. It is known that $F_{\ast}$
preserves the stability of vector bundles (\cite{S1}), but $F^{\ast}$ does not preserve
the semistability of vector bundle (\cite{DG} for example).

Given a semistable vector bundle $E$ on $X$, then $F^{\ast}E$ may
not be semistable, it's natural to consider the length of the
Harder-Narasimhan filtration of $F^{\ast}E$. If $E=F_{\ast}W$ where
$W$ is a semistable bundle on $X$, the length of Harder-Narasimhan
filtration of $F^{\ast}E$ is $p$ by Sun's Theorem (cf.
\cite[THeorem~2.2]{S1} is $p$. In this case, the instability of
$F^{\ast}E$ is $(p-1)(2g-2)$ which reaches the upper bound of
$F^{\ast}E$. Actually, in \cite{VC}, the authors prove the
instability $\mathrm{I}(F^{\ast}E)\leq (p-1)(2g-2)$ for a semistable
bundle $E$ on $X$. In consideration of the above statement, one may
guess that $\forall$ semistable bundle $E$ on $X$, the length of the
Harder-Narasimhan filtration of $F^{\ast}E$ is not more than $p$.

In this note, we give a negative answer to the above guess by
constructing an example. Briefly, give a polarized smooth projective
surface ($Y, H$) with $\mu(\Omega^1_Y)>0$ is semistable, for a
suitable semistable vector bundle $W$ on $Y$ (take $W$ to be line
bundle for example), $F_{\ast}W$ is semistable and
$F^{\ast}F_{\ast}W$ has the Harder-Narasimhan filtration whose
length is bigger than $p$. Restricting $F^{\ast}F_{\ast}W$ and the
Harder-Narasimhan filtration to a generic hyperplane $X$ of high
degree, by some analysis, one can check it's actually the
Harder-Narasimhan filtration of $F^{\ast}(E|_X)$, where
$E=F_{\ast}W$ and the stability of $E|_{X}$ would been proved in the
text.

\section{Construction of the example}
Let $Y$ be a smooth projective variety of dimension $n$ over an
algebraic closed field $k$ with $char(k)=p>0$. Fix an ample divisor
$H$ on $Y$, for a torsion free sheaf $E$, the slope of $E$ is
defined as $$\mu(E)=\frac{c_1(E)\cdot H^{n-1}}{rk(E)}$$ where
$rk(E)$ denotes the rank of $E$. Then
\begin{defn} A torsion free sheaf $E$ on $Y$ is called semistable (resp.
stable) if for any subsheaf $E'\subset E$ with $E/E"$ torsion free,
we have $$\mu(E')\leq (resp. <) \mu(E).$$
\end{defn}
\begin{thm} (Harder-Narasimhan filtration) For any torsion free
sheaf $E$, there is a unique filtration $$HN_{\bullet}(E):
0=HN_0(E)\subset HN_1(E)\subset \cdots \subset HN_{\ell}(E)=E'$$
which is the so called Harder-Narasimhan filtration, such that

\noindent (1) $gr_i^{HN}(E):= HN_i(E)/HN_{i-1}(E) (1\leq i\leq
\ell)$ are semistable;

\noindent (2) $\mu(gr_1^{HN}(E))>\mu(gr_2^{HN}(E))>\cdots
>\mu(gr_{\ell}^{HN}(E))$.
\end{thm}
\begin{rmk} In \cite[Theorem~1.3.4]{DM}, the proof of existence of
the filtration is given in terms of Gieseker stability. In
particular, $gr_i^{HN}(E)$ are Gieseker semistable, thus they are
$\mu$ semistable torsion free sheaves. We call the integer $\ell$ is
the length of the Harder-Narasimhan filtration.
\end{rmk}
\begin{defn} Let $E$ be a semistable sheaf. A Jordan-H\"{O}lder
filtration of $E$ is a filtration $$0=E_0\subset E_1\subset\cdots
\subset E_{\ell}=E$$ such that the factors $gr_i(E)=E_i/E_{i-1}$ are
stable .
\end{defn}
\begin{rmk} Jordan-H\"{o}lder filtration always exists. As in Remark
2.3, we can get the factors $gr_i(E)$ are torsion free.
\end{rmk}
Let $F: Y\rightarrow Y_1:= Y\times_k k$ denote the relative
Frobenius morphism over $k$.
\begin{lem} Let $X$ be a closed subvariety of $Y$, then the induced
morphism $F|_C: C\rightarrow C_1$ is the relative Frobenius
morphism.
\end{lem}
\begin{proof} For the case of absolute Frobenius morphism, it's
trivial. In the relative case, it's just a translation of the
absolute case to the relative case.
\end{proof}

Now, let's restrict $Y$ to be a smooth projective variety of
dimension $n$ with $\Omega^1_Y$ is semistable and
$\mu(\Omega_Y^1)>0$ (For example, $Y$ can been chosen to be $C\times
C$ with $C$ is a smooth projective curve with genus $g\geq 2$). Let
$W$ be a torsion free sheaf on $Y$, define
$V_0:=V=F^{\ast}(F_{\ast}W)$, $V_1=
ker(F^{\ast}(F_{\ast}W))\twoheadrightarrow W$,
$$V_{\ell+1}:= ker((V_{\ell})\xrightarrow{\nabla}V\otimes_{\sO_Y}\Omega^1_Y\rightarrow (V/V_{\ell})\otimes_{\sO_Y}\Omega^1_Y)$$
where $\nabla: V\rightarrow V\otimes_{\sO_Y}\Omega^1_Y$ is the
canonical connection. Actually, for the above filtration we have
\begin{thm} (\cite[Theorem~3.7]{S1}) The filtration defined above is
$$0=V_{n(p-1)+1}\subset V_{n(p-1)}\subset\cdots \subset V_1\subset V_0=V=F^{\ast}(F_{\ast}W)$$ which has the following properties

\noindent (1) $\nabla(V_{\ell+1})\subset V_{\ell}\otimes \Omega_Y^1$
for $\ell\geq 1$, and $V_0/V_{\ell}\cong W$.

\noindent (2)
$V_{\ell}/V_{\ell+1}\xrightarrow{\nabla}(V_{\ell-1}/V_{\ell})\otimes
\Omega_Y^1$ are injective for $1\leq \ell n(p-1)$, which induced
isomorphisms $$\nabla^{\ell}: V_{\ell}/V_{\ell+1}\cong
W\otimes_{\sO_Y} T^{\ell}(\Omega_Y^1), \ \  0\leq \ell \leq n(p-1)$$
The vector bundle $T^{\ell}(\Omega_Y^1)$ is suited in the exact
sequence
$$\aligned &0\rightarrow Sym^{\ell-\ell(p)\cdot
p}(\Omega_Y^1 )\otimes
F^{\ast}\Omega_Y^{\ell(p)}\xrightarrow{\phi}Sym^{\ell-(\ell(p)-1)\cdot
p}(\Omega^1_Y)\otimes F^{'\ast}\Omega_Y^{\ell(p)-1}\\
&\rightarrow\cdots \rightarrow Sym^{\ell-q\cdot
p}(\Omega_1^Y)\otimes
F^{\ast}\Omega_Y^{q}\xrightarrow{\phi}Sym^{\ell-(q-1)\cdot
p}(\Omega^1_Y)\otimes F^{'\ast}\Omega_Y^{q-1}\\
&\rightarrow \cdots \rightarrow Sym^{\ell- p}(\Omega_1^Y)\otimes
F^{\ast}\Omega_Y^{1}\xrightarrow{\phi}Sym^{\ell
}(\Omega^1_Y)\rightarrow  T^{\ell}(\Omega^1_Y)\rightarrow 0\\
\endaligned$$
where $\ell(p)\geq 0$ is the integer such that $\ell-\ell(p)\cdot p
<p$ .
\end{thm}
\begin{lem} If $\Omega_Y^1$ is strongly semistable with
$\mu(\Omega_Y^1)>0$ and $W$ is strongly semistable, then the
filtration defined in Theorem 2.7
$$0=V_{n(p-1)+1}\subset V_{n(p-1)}\subset \cdots
\subset V_1\subset V_0=V=F^{\ast}(F_{\ast}W)$$ is the
Harder-Narasimhan filtration of $F^{\ast}(F_{\ast}W)$.
\end{lem}
\begin{proof} From the Lemma 4.5 of \cite{S2}, one can get
$$\mu(V_{\ell}/V_{\ell+1})=\frac{l}{n}K_Y\cdot H^{n-1}+\mu(W).$$ So
the rest is enough to show that $W\otimes_{\sO_Y}
T^{\ell}(\Omega^1_Y)$ is semistable. By $W$ is strongly semistable,
it's enough to show $T^{\ell}(\Omega_Y^1)$ is strongly semistable.
In the long exact sequence of Theorem 2.7, every terms have slope
$\ell\cdot \mu(\Omega_Y^1)$  by direct computations and strongly
semistable by easily checking. Now, the lemma comes from the
following fact: For a short exact sequence of torsion free sheaves
with the same slope:
$$0\rightarrow E_1\rightarrow E_2\rightarrow E_3\rightarrow 0,$$
then one of them is strongly semistable if the other two is strongly
semistable.
\end{proof}

\begin{rmk} Under the above condition, one can deduce $F_{\ast}W$ is
semistable, denote it by $E$. Actually, it a direct corollary of the
following Theorem.
\end{rmk}

\begin{thm}(\cite[Theorem~4.2]{S2}) When $K_Y\cdot H^{n-1}\geq 0$,
we have, for any $E\subset F_{\ast}W$,
$$\mu(F_{\ast}W)-\mu(E)\geq -\frac{\mathrm{I}(W,X)}{P}.$$ In
particular, if $W\otimes T^{\ell}(\Omega_Y^1)$, $0\leq \ell \leq
n(p-1)$, are semistable, then $F_{\ast}W$ is semistalbe. Moreover,
if $K_Y\cdot H^{n-1}>0$, the stability of the bundles $W\otimes
T^{\ell}(\Omega^1_Y)$, $0\leq \ell \leq n(p-1)$, implies the
stability of $F_{\ast}W$.
\end{thm}

\begin{lem}(\cite[Corollary~5.4]{AL}) Let $E$ be a torsion-free
sheaf of rank $r>1$ on $Y$. Assume that $E$ is $\mu$ semistable with
respect to $(D_1,\cdots, D_{n-1})$ and let $0=E_0\subset E_1\subset
\cdots \subset E_m=E$ be the corresponding Jordan-H\"{O}lder
filtration of $E$, set $F_i=E_i/E_{i-1}$, $r_i=rk F_i$. Let $D\in
|kD_1|$ be a normal divisor such that all the sheaves $F_i|_D$ have
no torsion. If $$k> \lfloor\frac{r-1}{r} \Delta(E)D_2\cdots D_{n-1}+
\frac{1}{dr(r-1)}+ \frac{(r-1)\beta_r}{dr}\rfloor$$ then $E|_D$ is
$\mu$ semistable with $(D_2|_D,\cdots, D_{n-1}|_D)$.
\end{lem}
\begin{rmk} For the notation, one can see \cite{AL} for detail. Note
that if $E$ is torsion free then the restriction $E|_D$ is also
torsion free for a general divisor $D$ in a base point-free system
(see \cite[Corollary~1.1.14]{DM} for a precise statement). Take
$k\gg 0$, using the Bertini's Theorem (cf. \cite[Theorem~8.8]{AG}),
we can choose $D\in |kD_1|$ to be smooth projective variety and
$E|_D$ is a semistable torsion free sheaf with respect to
$(D_2|_D,\cdots, D_{n-1}|_D)$. Repeating the above process, we can
get a smooth projective curve $X$ which is a subvariety of $Y$, such
that $E|_C$ is a semistable vector bundle. From our choice, the
curve $C$ has large genus.
\end{rmk}

Now, we can construct the example using the above preliminaries.
\begin{construction} Let $Y$ to be a smooth projective variety of dimension $n$ with
$\Omega^1_Y$ is strongly semistable and $\mu(\Omega_Y^1)>0$, for
example we can take $Y=C\times\cdots\times C$ with $C$ is a smooth
curve of genus $G\geq 2$. Let $W$ to be a strongly semistable bundle
on $Y$, for example we can take $W$ to be the copies of line bundle.
Consider the following diagram:

\[
\begin{CD}
Y @>F>>
Y\\
@AAA @AAA\\
X@>F>>X
\end{CD}
\]
where $X$ is smooth projective curve which is chosen as in Remark
2.13, and the commutative diagram comes from Lemma 2.6. Denote
$F_{\ast}W$ by $E$, then $E$ is semistable from Remark 2.9 and
$E|_X$ is semistable from Remark 2.12. Denote the Harder-Narasimhan
filtration of $F^{\ast}E$
$$0=V_{n(p-1)+1}\subset V_{n(p-1)}\subset \cdots
\subset V_1\subset V_0=V=F^{\ast}(F_{\ast}W)$$ by $\bf{HN}$, we have
the length of $\bf{HN}$ is $n(p-1)+1$. Consider the restriction of
the filtration $HN$ to $X$ which denoted by $\bf{HN}|_X$,
$(V_i/V_{i+1})|_X=(V_i|_X)/(V_{i+1}|_X)$ is semistable by Remark
2.12 and  $\mu((V_i|_X)/(V_{i+1}|_X))$ is strictly increasing.
Meanwhile, $(F^{\ast}E)|_X=F^{\ast}(E|_X)$ by the commutative
diagram and $\bf{HN}|_X$ is the Harder-Narasimhan filtration of
$F^{\ast}(E|_X)$. Above all, we construct a semistable bundle $E|_X$
over a smooth curve $X$ with genus $g>2$, such that the length of
the Harder-Narasimhan filtration is $n(p-1)+1$. So there is no bound
for the length of the Harder-Narasimhan filtration of a bundle as
Frobenius pull-back of a semistable bundle.
\end{construction}

\bibliographystyle{plain}

\renewcommand\refname{References}

\end{document}